\documentclass[12pt, a4paper]{amsart}
\usepackage{breqn}
\usepackage{pdflscape}
\usepackage{array}
\usepackage{fullpage}
\usepackage[utf8]{inputenc}
\usepackage{amssymb}
\usepackage{amsthm}
\usepackage{amsmath}
\usepackage{multirow}

\usepackage{footnotehyper}
\makesavenoteenv{tabular}
\makesavenoteenv{table}

\usepackage[textsize=tiny]{todonotes}
\usepackage{color}
\definecolor{webcolor}{rgb}{0,0,1}
\definecolor{webbrown}{rgb}{.6,0,0}
\usepackage[
        colorlinks,
        linkcolor=webbrown, filecolor=webbrown,  citecolor=webbrown,
        pdfauthor={},
  pdftitle={},
]{hyperref}

\DeclareMathOperator{\CC}{C}
\DeclareMathOperator{\DD}{D}

\DeclareMathOperator{\QQQ}{Q}

\DeclareMathOperator{\Dic}{Dic}
\DeclareMathOperator{\ccenter}{Z}
\DeclareMathOperator{\norm}{N}

\newcommand{\Z}{\mathbf{Z}} 
\newcommand{\Q}{\mathbf{Q}}
\newcommand{\C}{\mathbf{C}}
\newcommand{\R}{\mathbf{R}} 
\newcommand{\F}{\mathbf{F}}

\theoremstyle{definition}
\newtheorem{introtheorem}{Theorem}[section]

\newtheorem{introcoro}[introtheorem]{Corollary}

\newtheorem{theorem}{Theorem}[section]

\newtheorem{lemma}[theorem]{Lemma}
\newtheorem{proposition}[theorem]{Proposition}
\newtheorem{corollary}[theorem]{Corollary}

\newtheorem{example}[theorem]{Example}
\newtheorem{remark}[theorem]{Remark}

\newtheorem*{theorem*}{Theorem}
\newtheorem*{corollary*}{Corollary}
\newtheorem*{notation*}{Notation}

\newcommand{\Gal}{\mathrm{Gal}}

\newcommand{\disc}{\mathrm{disc}}

\newcommand{\CCl}{\mathrm{Cl}}
\newcommand{\D}{\mathrm{D}}

\newcommand{\Alt}{\mathfrak A}
\newcommand{\GL}{\mathrm{GL}}

\newcommand{\SL}{\mathrm{SL}}

\newtheorem{innercustomgeneric}{\customgenericname}
\providecommand{\customgenericname}{}
\newcommand{\newcustomtheorem}[2]{%
  \newenvironment{#1}[1]
      {%
        \renewcommand\customgenericname{#2}%
              \renewcommand\theinnercustomgeneric{##1}%
                 \innercustomgeneric
                   }
                     {\endinnercustomgeneric}
                     }

                     \newcustomtheorem{customproblem}{Problem}

\author{Tommy Hofmann}
\address{Fakultät fur Mathematik und Informatik, Universität des Saarlandes, 66123 Saarbrücken, Germany}
\email{thofma@gmail.com}

\author{Carlo Sircana}
\address{Carlo Sircana\\
Fachbereich \ Mathematik\\
Technische \  Universität \  Kaiserslautern\\
67663 Kaiserslautern\\
Germany}
\email{sircana@mathematik.uni-kl.de}

\makeatletter
\@namedef{subjclassname@2020}{%
  \textup{2020} Mathematics Subject Classification}
\makeatother

\date{\today}

\makeatletter
\providecommand\@dotsep{5}
\def\listtodoname{List of Todos}
\def\listoftodos{\@starttoc{tdo}\listtodoname}
\makeatother

\title{Normal CM-fields with class number one}

\begin{document}

\maketitle

\begin{abstract}
  We show that assuming the generalized Riemann hypothesis there are no normal
  CM-fields with class number one of degree 64 and 96. This is done by constructing complete tables
  of normal CM-fields using discriminant bounds of Lee--Kwon.
  This solves the class number one problem for normal CM-fields assuming GRH.
  Using the same technique to solve the relative class number one problem in degrees 16, 32, 56 and 82,
  also the corresponding relative class number one problem is solved assuming GRH.
\end{abstract}

\section{Introduction}

A number field $L$ is called a \textit{CM-field}, if $L$ is a totally imaginary quadratic extension of its maximal totally real subfield $L^+$.
It follows from class field theory, that the class number $h_{L^+}$ divides $h_{L}$.
The number $h_{L}^- = h_L/h_{L^+}$, which is a divisor of $h_L$, is called the \textit{relative class number} of $L$.
Let us denote by $d$ the degree of $L$.
In 1974, in the seminal work~\cite{Odlyzko1975} it was shown by Odlyzko that there are only finitely many normal
CM-fields with a given class number, thus giving rise to the \textit{(relative) class number one problem} for CM-fields.
This problems asks for the determination of all normal CM-fields with (relative) class number one.
Note that this generalizes Gauß famous class number one problem, which asks for a list of all imaginary quadratic number fields with class number one and which was solved by Baker~\cite{Baker1966} and Stark~\cite{Stark1967} (see also~\cite{Goldfeld1985, Stark2007}).

An explicit upper bound for the degree $d$ of such a normal CM-field with relative class number one was proven in 1979 by Hoffstein, who showed that such a field must satisfy $d \leq 434$.
In 2003 this was improved by Bessassi~\cite{Bessassi2003} to $d \leq 266$ (and $d \leq 164$ assuming the Generalized Riemann Hypothesis (GRH)).
The presently best known bounds are due to Lee--Kwon~\cite{Lee2006}, who in 2006 showed that $d \leq 216$ (and $d \leq 96$ assuming GRH).
For abelian CM-fields the class number one problem was solved in 1994 by Yamamura~\cite{Yamamura1994} and the relative class number one problem in 2000 by Chang--Kwon~\cite{Chang2000}.
Since 1994, solving the class number one problem in the non-abelian case for specific degrees and Galois groups has been a major undertaking by various authors.
The case $d = 8$ was settled by Louboutin--Okazaki~\cite{Louboutin1994}, the case $d = 12$ by Louboutin--Okazaki--Olivier~\cite{Louboutin1997} and the case $d = 16$ by Louboutin~\cite{Louboutin1997b} and Louboutin--Okazaki~\cite{Louboutin1998}.
For fields of degree $4p$, $p > 3$ a prime, or with Galois dihedral or dicyclic, the problem was solved by Louboutin~\cite{Louboutin1999}, Lefeuvre--Louboutin~\cite{Lefeuvre1998} and Lefeuvre~\cite{Lefeuvre2000}.
For $d \in \{20,40\}$, the work of~\cite{Lemmermeyer1999} and Park~\cite{Park2002} solves the class number one problem.
For $d = 36$ the problem was solved by Chang--Kwon~\cite{Chang2002} and for $d = 48$ by Chang--Kwon~\cite{Chang2003} and Park--Kwon~\cite{Park2007}.
The case $d = 32$ was settled by Park--Yang--Kwon~\cite{Park2007}.
Finally, for $d \in \{1,\dotsc,96\} \setminus \{64, 96\}$, the remaining cases were settled by Park--Kwon~\cite{Park2007}.
Thus assuming GRH, except for the possible degrees 64 and 96, all normal CM-fields with class number one have been determined.

In the present paper, we settle the missing cases by showing the following:

\begin{introtheorem}\label{thm:intro}
  If we assume the Generalized Riemann Hypothesis, then for a normal CM-field $L$ of degree 64 or 96 we have $h_L^- > 1$.
\end{introtheorem}

Combining this with the results cited in the previous paragraph we obtain the following classification:

\begin{introcoro}
  If we assume the Generalized Riemann Hypothesis, there are 227 normal CM-fields with class number one.
  The number of these CM-fields for a given degree and Galois group are given in Tables~\ref{tab:allthefieldsdegrees} and~\ref{tab:allthefieldsgalois} respectively.
  The fields themselves are listed in Appendix~\ref{appendix:fields}.\footnote{The table of fields is also available on the homepage of the first author.}
\end{introcoro}

\begin{table}[htbp]\caption{The number of normal CM-fields with (relative) class number one and given degree.}
    \centering
    \begin{tabular}{|c|c|c|c|c|c|c|}\hline \label{tab:allthefieldsdegrees}
      \rule{0pt}{2.0ex}\multirow{2}{*}{Degree} & \multicolumn{2}{c|}{All} & \multicolumn{2}{|c|}{Abelian} & \multicolumn{2}{|c|}{Non-abelian} \\ \cline{2-7}
       \rule{0pt}{2.5ex}& $\#\{h_L = 1\}$ & $\#\{h_L^- = 1\}$ & $\#\{h_L = 1\}$ & $\#\{h_L^- = 1\}$ & $\#\{h_L = 1\}$ & $\#\{h_L^- = 1\}$  \\
      \hline
 \rule{0pt}{1.5ex}2      &  9    &  9    &  9   & 9   &  0    &  0   \\
                  4      &  54   &  154  &  54  & 154 &  0    &  0   \\
                  6      &  17   &  26   &  17  & 26  &  0    &  0   \\
                  8      &  54   &  62   &  37  & 43  &  17   &  19  \\
                  10     &  3    &  3    &  3   & 3   &  0    &  0   \\
                  12     &  35   &  56   &  26  & 40  &  9    &  16  \\
                  14     &  2    &  2    &  2   & 2   &  0    &  0   \\
                  16     &  24   &  26   &  12  & 13  &  12   &  13  \\
                  18     &  3    &  3    &  3   & 3   &  0    &  0   \\
                  20     &  5    &  6    &  4   & 4   &  1    &  2   \\
                  24     &  12   &  12   &  5   & 5   &  7    &  7   \\
                  32     &  4    &  4    &  0   & 0   &  4    &  4   \\
                  36     &  3    &  3    &  0   & 0   &  3    &  3   \\
                  40     &  1    &  1    &  0   & 0   &  1    &  1   \\
                  48     &  1    &  1    &  0   & 0   &  1    &  1   \\
 \hline
 \rule{0pt}{2.0ex}$\Sigma$ & 227 & 368 & 302 & 172 & 66 & 55 \\ \hline
 \end{tabular}\vspace{1em}
 \end{table}

\begin{table}[htbp]\caption{Number of normal CM-fields with (relative) class number one and given Galois group.}
   \begin{tabular}{|c|c|c|c||c|c|c|c|}\hline \label{tab:allthefieldsgalois}
     \rule{0pt}{2.5ex} $\Gal(L/\Q)$ & ID & $h_L = 1$ & $h_L^- = 1$ & $\Gal(L/\Q)$ & id & $h_L = 1$ & $h_L^- = 1$  \\
      \hline
     \rule{0pt}{1.5ex}$\CC_2$ & (2, 1) &  9  &  9 & %
     $\DD_4 \rtimes \CC_2$ & (16, 13) & 1     & 1   \\ %
     $\CC_4$ & (4, 1)  & 7     & 7   & %
     $\CC_{18}$  & (18, 2) & 2     & 2   \\ %
     $\CC_2^2$   & (4, 2)& 47    & 147 & %
     $\CC_3 \times \CC_6$  & (18, 5) & 1     & 1   \\ %
     $\CC_6$ & (6, 2)  & 17    & 26  & %
     $\CC_{20}$ & (20, 2) & 1     & 1   \\ %
     $\CC_8$ & (8, 1)  & 2     & 2   & %
     $\DD_{10}$ & (20, 4) & 1     & 2   \\ %
     $\CC_2 \times \CC_4$ & (8, 2) & 18    & 24  & %
     $\CC_2 \times \CC_{10}$ & (20, 5) & 3     & 3    \\ %
     $\DD_4$ & (8, 3)  & 17    & 19  & %
     $\SL_2(\F_3)$ & (24, 3) & 1     & 1   \\ %
     $\CC_2^3$ & (8, 5) &  17    & 17  & %
     $\CC_4 \times \mathfrak S_3$ & (24, 5) & 1     & 1  \\ %
     $\CC_{10}$ &(10, 2)  & 3     & 3   & %
     $\DD_{12}$ & (24, 6) & 1     & 1   \\ %
     $\CC_{12}$ & (12, 2) & 5     & 6   & %
     $\CC_2 \times \CC_{12}$ &(24, 9) & 3     & 3   \\ %
     $\DD_6$ & (12, 4) & 9     & 16  & %
     $\CC_3 \times \DD_4$ & (24, 10) & 1     & 1   \\ %
     $\CC_2 \times \CC_6$ & (12, 5)  & 21    & 34  & %
     $\CC_2 \times \mathfrak A_4$ & (24, 13) & 2     & 2   \\ %
     $\CC_{14}$ & (14, 2) & 2     & 2   & %
     $\CC_2^2 \times \mathfrak S_3$ & (24, 14) & 1     & 1   \\ %
     $\CC_{16}$ & (16, 1)  & 1     & 1   & %
     $\CC_2^2 \times \CC_6$ & (24, 15) & 2     & 2   \\ %
     $\CC_4^2$ & (16, 2) & 2     & 2   & %
     $\CC_2^2 \rtimes \CC_8$ & (32, 5) & 1     & 1   \\ %
     $\CC_2^2 \rtimes \CC_4$ & (16, 3) & 2     & 2   & %
     $\CC_2 \times \CC_4 \rtimes \CC_4$ & (32, 23) & 1     & 1   \\ %
     $\CC_2 \times \CC_8$ & (16, 5) & 2     & 3   & %
     $\CC_2 \times (\CC_4 . \CC_4)$ & (32, 37) & 1     & 1   \\ %
     $\DD_8$ & (16, 7) & 4     & 5   & %
     $\CC_4 \circ \DD_4$ & (32, 42) & 1     & 1   \\ %
     $\CC_2^2 \times \CC_4$ & (16, 10) & 7     & 7   & %
     $\mathfrak S_3 \times \CC_6$ & (36, 12) & 3     & 3   \\ %
     $\CC_2 \times \DD_4$ & (16, 11) & 4     & 4   & %
     $\CC_2 \times (\F_5 \rtimes \F_5^\times)$ & (40, 12) & 1     & 1   \\ %
     $\CC_2 \times \QQQ_8$ & (16, 12) & 1     & 1   & %
     $\DD_4 \rtimes^* \mathfrak S_3$\footnotemark & (48, 15) & 1     & 1   \\\hline %
 \end{tabular}
 \end{table}

The discussion so far has been concerned with the class number one problem. In contrast, for the \textit{relative} class number one problem, the classification has been less complete.
While for abelian CM-fields the problem has been solved by Chang--Kwon~\cite{Chang2000}, the non-abelian situation is much more involved. Indeed, results of various authors for specific degrees and Galois groups cited in the previous paragraph only sometimes settle also the relative class number one problem (see Section~\ref{sec:relcmone} for details).
We work out the remaining cases for the relative class number one problem and obtain the following result.

\begin{introcoro}\label{intro:coro}
  If we assume the Generalized Riemann Hypothesis, then there are 368 normal CM-fields with relative class number one. 
  The number of these CM-fields for a given degree and Galois group are given in Tables~\ref{tab:allthefieldsdegrees} and~\ref{tab:allthefieldsgalois}.
  The fields themselves are listed in Appendix~\ref{appendix:fields}.
\end{introcoro}

Let us now explain the strategy for proving Theorem~\ref{thm:intro} (as well as Corollary~\ref{intro:coro}).
We will now assume that GRH holds.
The work of Lee--Kwon~\cite{Lee2006} not only provides the degree bound of 96, but also discriminant bounds of the following form:
There exists an efficient algorithm which given any integer $d \geq 10$ determines $\alpha(d) \in \R_{>0}$ such that a normal CM-field $L$ of degree $d$ with $h_L^- = 1$ satisfies $\lvert \disc(L)\rvert \leq \alpha(d)^n$.
Moreover, for various $d$, the values $\alpha(d)$ are determined, for example, $\alpha(64) = 61.90$, $\alpha(96) = 50.71$ and $\alpha(10) = 1163$.
To solve the relative class number one problem, we will determine the normal CM-fields in this range explicitly and then discard those fields with $h_L > 1$ and $h_L^- > 1$ respectively.

Since there are only finitely many number fields with
bounded absolute discriminant and these fields can be
determined explicitly using a method originally due to
Hunter~\cite{Hunter1957} (see also~\cite{Martinet1985}),
this shows that the (relative) class number one problem is decidable.
On the other hand, for degrees as large as 96 and the discriminant bounds provided by Lee--Kwon, these classical methods are just impractical. For example, in~\cite{Voight2008}, the computation of all totally real fields of degree $10$ with root discriminant bounded by $14$ with this method is described, which required half a CPU year.

\footnotetext[2]{This group is isomorphic to $\langle x, y, z \mid x^2, y^2, z^2, y z y z^{-1}, (xz)^2, (xy)^8\rangle$, which is \textit{a} semidirect product of $\DD_4$ and $\mathfrak S_3$.}

To make the computation of all normal CM-fields in this parameter range feasible, we will employ a different strategy.
For a normal CM-field with Galois group $G$, complex conjugation induces a non-trivial involution in the center of $G$. Thus $G$ is not equal to $\mathfrak A_5$ and in particular, if $d = \lvert G \rvert \leq 96$, then $G$ must be solvable.
Hence, any field $L$ we have to investigate is solvable and therefore can be constructed as a tower of abelian extensions
\[ \Q = L_0 \subsetneq L_1 \subsetneq \dotsb \subsetneq L_l = L \] 
using computational class field theory (see \cite{Cohen2008}).
Again, doing this naively is impractical, as one will construct too many fields which are either non-normal or which have the wrong Galois group.
To make this practical, we use a normal series of $G$ and the methods of \cite{Fieker2019} to construct only normal fields at each step.
Additionally, we sieve the intermediate fields $L_i$ using a theorem of Horie--Okazaki, which asserts that if $L_i$ is a CM-field and $h_{L_i}^- \not\in \{1, 2, 4\}$, then $h_L^- > 1$ (see Theorem~\ref{thm:witness}).
For details and other improvements (which do not work for all groups $G$), see Section~\ref{sec:construction}.
Finally we determine which of the fields $L$ satisfy $h_L^- = 1$ and $h_L = 1$ respectively. Note that as class number computations for degrees as large as 96 are out of range of current algorithms and implementations from a practial point of view, we again make use of the result Horie--Okazaki:
For the fields under consideration we will find CM-subfields with relative class number $>4$, which suffices to show that the fields themselves have relative class number greater than one.

We have used this strategy to solve the relative class number one problem for degree 64 and 96 to obtain Theorem~\ref{thm:intro}, as well as for the missing degrees and groups to obtain Corollary~\ref{intro:coro}.
In terms of a single CPU core, the runtime of the implementation amounts to 280 CPU days, which using parallelization can be finished within a week on a modern cluster.
To validate our implementation (and to obtain defining equations for all normal CM-fields with relative class number one), we have used the same strategy to recomputed all the normal CM-fields with degree $10 \leq d \leq 96$.
Except in two cases, we obtained results in agreement with the literature. In degree 12 we discovered a typo in one of the published tables and in degree 32, two fields were erroneously reported (see Remark~\ref{rem:32}).

From the description of our strategy it should be clear that the results we obtain rely heavily on explicit calculations.
We want to stress that to a large extend the same is true for previous results on the classification, which dealt with explicit degrees or Galois groups.
A common strategy is to use the subfield structure of $L$ or analytic techniques pioneered by Louboutin in~\cite{Louboutin1998b}, to obtain lower bounds on the relative class number of $L$.
This is then used to build a small list of candidate fields $L$, for which the relative class number is computed using knowledge about the Hasse unit index of $L$ and an algorithm of Louboutin (see \cite{Louboutin2000, Louboutin2001}).
For example, in~\cite{Park2007} this strategy is used to determine in an ad-hoc fashion for each of the 44 non-abelian groups of order 32 the normal CM-fields with (relative) class number one.

\subsection*{Outline}
The paper is built up as follows. In Section~\ref{sec:prelim} we recall basic facts about CM-fields and their relative class numbers.
The efficient construction of normal CM-fields with given solvable Galois group is discussed in Section~\ref{sec:construction}.
In Sections~\ref{sec:64} and~\ref{sec:96} we present Theorem~\ref{thm:intro} and in Section~\ref{sec:relcmone} we consider the relative class number one problem.
We present some details of the computation in Section~\ref{sec:comp} and end the paper with Appendix~\ref{appendix:fields}, which contains defining equations for all normal CM-fields with (relative) class number one.

\subsection*{Acknowledgments}
The authors gratefully acknowledge financial support by SFB-TRR 195 ``Symbolic Tools in Mathematics and their Application'' of the
German Research Foundation (DFG).

\subsection*{Notation}
The center of a group $G$ will be denoted by $\ccenter(G)$ and the derived subgroup by $G'$.
For $n \in \Z_{>0}$, we denote by $\CC_n$, $\DD_n$, $\Dic_{n}$, $\QQQ_{2^n}$ $\mathfrak S_n$ and $\mathfrak A_n$ the cyclic group of order $d$, the Dihedral group of order $2n$, the dicyclic group of order $4n$, the quaternion group of order $2^n$, the symmetric group and the alternating group on $n$ symbols respectively.
For groups $G$ and $H$ we will denote by $G \times H$ and $G \wr H$ the direct product and wreath product respectively.
In case it exists and is unique we will denote by $G \circ H$, $G \rtimes H$ and $G.H$ the central product, the split extension and the non-split extension of $H$ by $G$ respectively. If the extension is not unique, we will indicate this by writing $G \rtimes^* H$ and $G ._* H$ respectively.
We will also use the identifier as provided in the table of small groups~\cite{Besche2002}.

For a number field $K$ we denote by $\mathcal O_K$ its ring of integers, by $\disc(K) \in \Z$ its discriminant and by $\rho(L) = \lvert \disc(K)\rvert^{1/n}$ its root discriminant.

\section{Preliminaries}\label{sec:prelim}

We collect some useful facts about (normal) CM-fields and their class numbers.
In the following, we consider number fields in a fixed algebraic closure $\overline \Q \subseteq \C$.

A finite extension $L$ of $\Q$ is called a \textit{CM-field}, if $L$ is totally imaginary quadratic extension of a totally real field.
Such a field clearly has even degree. Given a CM-field $L$ of degree $d$, the unique totally real subfield of degree $d/2$ is called the \textit{maximal real subfield} of $L$ and is denoted by $L^+$.
Moreover, any subfield of a CM-field is either a CM-field or totally real.
If $L$ is normal field, complex conjugation $\C \to \C$ restricts to an automorphism of $L$, which we denote by $c$.
We have the following group theoretical characterization of normal CM-fields, see for example~\cite[Lemma 2~(ii)]{Louboutin1997}.

\begin{lemma}
  A totally imaginary normal field $L$ with Galois group $G$ is a CM-field if
  and only if complex conjugation $c$ lies in the center $\ccenter(G)$ of $G$.
\end{lemma}

Thus, not every finite group is the Galois group of a normal CM-field. In fact, using the
previous result one obtains the following restrictions on degrees and Galois groups (see for
example~\cite[Lemma 2~(iv)]{Louboutin1997}).

\begin{lemma}
  Assume that $L$ is a normal CM-field of degree $d$. Suppose that $L$ is not abelian. Then $d \not\in \{4, 90\}$.
  Moreover $d$ is not of the form $2p$ or $2p^2$ for an odd prime $p$.
\end{lemma}

We also have the following important result on the normality of the maximal real subfield.

\begin{corollary}\label{cor:realnormal}
  If $L$ is a normal CM-field, then the maximal real subfield $L^+$ is also normal.
\end{corollary}

For a number field $L$ denote by $h_L$ the class number of $L$, that is, the
order of the class group of the ring of integers $\mathcal O_L$ of $L$.
Let $L$ be a CM-field.
It follows from an application of class field theory that the class number
$h_{L^+}$ of the totally real subfield divides divides the class number $h_L$
of $L$ (see~\cite[Theorem 4.10]{Washington1997}).
We denote this quotient by $h_L^- = h_L/ h_{L^+}$ and call it the \textit{relative class number} of $L$.
The following behavior of relative class numbers with respect to subfields is a key ingredient in the construction of normal CM-fields with relative class number one.
The statement was proven for abelian fields by Horie~\cite{Horie1992} and for arbitrary CM-fields by Okazaki~\cite{Okazaki2000}.

\begin{theorem}[{\cite[Theorem 1, Corollary 29]{Okazaki2000}}]\label{thm:witness}
  For two CM-fields $k \subseteq L$ we have $h_k^- \mid 4h_L^-$. If $[L : k]$ is odd, then also $h_k^- \mid h_L^-$ holds.
\end{theorem}

\section{Constructing normal CM-fields}\label{sec:construction}

Throughout this section, we will consider a finite solvable group $G$ and a positive number $B \in \R_{>0}$.
We want to describe the construction of all normal CM-fields $L$ inside a fixed
algebraic closure $\overline \Q$ such that $\Gal(L/\Q) \cong G$ and $\lvert
\disc(L) \rvert \leq B$.
Since $G$ is solvable, there exists a normal series with abelian quotients
\begin{align}\label{eq:subnormal}
1 = G_0 \subsetneq G_1 \subsetneq \dotsc \subsetneq G_l = G,\tag{$\star$}
\end{align}
that is, $G_i$ is normal in $G$ and $G_{i+1}/G_i$ is abelian for $0 \leq i \leq l-1$.
This implies that if $L$ is a normal extension with Galois group $G$, then there exists a tower of subfields
\[ 
\Q = L_0 \subsetneq L_{1} \subsetneq \dotsb \subsetneq L_l = L,
\]
where $L_i = L^{G_{i - i}}$ and thus $\Gal(L/L_{i}) = G_{i-1}$ for $0 \leq i \leq l$.
Since each layer $L_i/L_{i - 1}$ is abelian with Galois group $\Gal(L_{i}/L_{i - 1}) \cong G_{l - i + 1}/G_{l - i}$, it suffices for our purpose to be able to construct abelian extensions of a given number field with bounded discriminant.

The main tool for constructing abelian extensions of a given number field is computational class field theory,
which, in a nutshell, provides us with the following result.
For an extension $L/K$ of number fields we denote by $\mathfrak d_{L/K} \subseteq \mathcal O_K$ the relative discriminant and
by $\norm(\mathfrak a) = \lvert \mathcal O_K/\mathfrak a \rvert$ the absolute norm of an ideal $\mathfrak a \neq \{0\}$ of $\mathcal O_K$.

\begin{theorem}\label{thm:cft}
  There exists an efficient algorithm that given
  an abelian group $A$, a number field $K$ and a positive number $B \in \R_{> 0}$,
  determines all extensions $L$ of $K$ with $\Gal(L/K) \cong A$ and $\norm(\mathfrak d_{L/K}) \leq B$.
\end{theorem}

The relative extensions computed are specified by a defining polynomial
equation, that is, an irreducible polynomial $g \in K[x]$ such that $L \cong K[x]/(g)$.
The term \textit{efficient} is not meant as a precise statement on the complexity,
but more of a judgment based on practical observations and experiments.
We will refrain from diving into all the details of this
topic and refer the reader to~\cite{Cohen2000, Cohen2008} as well the recent
improvements in~\cite{Fieker2019} (in particular, on how to exploit the
normality of the chain). We will comment on one particular subproblem.

\begin{remark}
  From class field theory it follows that in Theorem~\ref{thm:cft}, the abelian
  extensions $L$ are parameterized by pairs $(\mathfrak f, A)$, consisting of
  the conductor of $L/K$ and a subgroup $A \subseteq \CCl_{\mathfrak f}$ of the
  ray class group such that $\CCl_{\mathfrak f}/A \cong G$ (modulo a certain
  equivalence relation).
  Although this allows us to enumerate the abelian extensions, it does not provide a defining polynomial equation.
  Finding the defining polynomial equation is the second step of the algorithm and in practice the most expensive one
  if the parameters, that is, the degree and discriminant of $K$ or the
  cardinality of $G$, are large.  The idea behind this step is to use Kummer
  theory. More precisely, one finds generators of a suitable $S$-unit group $U$
  of $K(\zeta_e)$ such that $K(\zeta_e) \subseteq L(\zeta_e) \subseteq
  K(\zeta_e, \sqrt[e]{U})$, where $e$ is the exponent of $G$ and $\zeta_e$
  denotes a primitive $e$-th root of unity.  In particular, this computation of
  an $S$-unit group of $K(\zeta_e)$, which is comparable to the computation of
  the class group of $K(\zeta_e)$, was until recently only possible using
  subexponential algorithms (see~\cite[\S{}I.1.2]{Simon1998}). We will
  make instead use of the recent algorithmic advantages for number fields with many
  subfields as introduced in~\cite{Biasse2019}. These are often applicable
  in our situation, since the field $K(\zeta_e)$ usually admits plenty
  of subfields with the right structure.
\end{remark}

Following the strategy just described, to find all normal CM-fields with Galois group $G$ and relative class number one, one could try to first construct all normal extensions with Galois group $G$. For this list of candidates one then needs to discard all fields which are not CM-field or whose relative class number is not one.
This approach has the disadvantage that the number of candidates is in practice too large, rendering this strategy futile.
Instead we apply the following improvements to keep the list of candidates manageable. In all cases, the idea is to apply aggressive sieving to the constructed subfields $L_i$. In this regard, the choice of normal series~\eqref{eq:subnormal} plays an important role.

\subsection{Choice of a series}\label{subsection:choice_series}
In most cases, we will use as a chain the derived series of the group $G$. The choice of this series minimizes the number of layers we need to construct. On the other hand, it does not exploit the fact that we are searching for CM-fields. For this reason, we sometimes use a different series.

    In view of Corollary~\ref{cor:realnormal} one option would be to first
    determine the possible maximal real subfields and then compute the
    possible CM-fields as totally imaginary quadratic extensions.  By checking
    tables of small groups, one can quickly identify the possible Galois
    groups $H_1,\dotsc,H_r$ of the maximal real subfield of a normal CM-field
    with Galois group $G$ (the group $G$ must be a central $\CC_2$-extension
    of the group $H_i$).
    When constructing normal CM-fields for \textit{all} possible Galois of a given degree,
    this approach has the advantage of computing every maximal real subfield only once.

\subsection{Control over the class number}\label{subsection:relative_class_number_cm}
  In the case we are using the derived series, we can exploit the fact that we are looking only for CM-fields with relative class number one, which implies that all CM-subfields must have relative class number equal to $1$, $2$ or $4$ (Theorem~\ref{thm:witness}).
    This allows to apply a very aggressive sieving when constructing the CM-fields iteratively:
    As soon as we have found a field $L_i$ which is a CM-field and which satisfies $h_{L_i}^- \notin \{1, 2, 4\}$,
    we can drop $L_i$ from the process.
    As we actually expect the proportion of CM-fields with relative class number one and Galois group $G$
    to be very small, this approach will be quite effective.
    (Note that this idea cannot be used when first constructing the possible maximal totally real subfields, since the CM-fields are constructed only in the very last step.)
    On the group theoretical
    side checking if a subfield is a CM-field can be predicted by noticing that for a subgroup $H \subseteq \Gal(L/\Q)$ the field $L^H$ is a CM-field if and only if
    complex conjugation is not contained in $H$. Thus, for example, if
    $H \cap \ccenter(G) = 1$, then $L^H$ must be a CM-field.
    \begin{example}
      Consider the group $G = \CC_2^3 ._* \DD_4$ with identifer $(64, 67)$. Assume that we want to find all normal CM-fields with Galois group $G$ and discriminant bounded by $10^{115}$.
      The abelianization of $G$ is isomorphic to $H = \CC_2^2\times \CC_4$ and there are $85$ fields $L_1$ with Galois group $H$ with discriminant bounded by $10^{115/4}$ that are either totally real or CM-fields satisfying $h_{L_1}^- \in \{1, 2, 4\}$.
      As a second step, we construct all the fields with Galois group isomorphic to $G$ that are extensions of one of these $85$ fields. Among the $85$ fields, only $13$ fields admit an extension with Galois group $G$ and discriminant bounded by $10^{115}$, giving rise to $57$ fields with Galois group $G$. None of these fields is a CM-field.

If we had not used the criterion above, the computation would have been much more complicated.
Indeed, we would have tried to extend all of the $642$ number fields with Galois group isomorphic to $H$ and discriminant bounded by $10^{115/4}$. Consequently, we would have found $188$ fields with Galois group isomorphic to $G$, of which only one is a CM-field with relative class number (of course) not one.
\end{example}
    
\subsection{Control over the ramification at the infinite places}\label{subsection:ramification_infinite}
    
Since we are searching for CM-fields, we want the complex conjugation $c$ to lie in the center of the Galois group of the target fields.
This implies that if the group $G_i$ in the normal series \eqref{eq:subnormal} contains the center $\ccenter(G)$, the $i$-th field in the tower of subfields must be totally real.
This imposes a constraint on the fields we construct, which can be exploited as in the following example.

\begin{example}
  Consider the construction of CM-fields with Galois group $\QQQ_{64}$. The derived series has length $2$ and the quotient of $\QQQ_{64}$ by its derived subgroup, which contains the center of the group, is isomorphic to $\CC_2^2$. Therefore, according to the criterion, the abelian extensions of $\mathbf Q$ with Galois group $\CC_2^2$ that extend to a field with Galois group $\QQQ_{64}$ have to be totally real.
  This dramatically improves the runtime: The number of biquadratic extensions of discriminant bounded by $10^{115/16}$ is $5460$, while the number of those fields which are totally real fields is $1141$.
    \end{example}

\section{Degree 64}\label{sec:64}

We now focus on the computation of normal non-abelian CM-fields of degree $64$ with relative class number one.
If $L$ is a normal CM-field of degree 64 with $h_L^- = 1$, then according
to~\cite[Theorem 1]{Lee2006} the root discriminant satisfies $\rho(L) \leq
61.90$, which implies $\lvert \disc(L) \rvert \leq 10^{115}$.
From~\cite{Besche2002} we know that there are 256 non-abelian groups of order 64.

We use the derived series together with the criteria given in Sections~\ref{subsection:relative_class_number_cm} and~\ref{subsection:ramification_infinite}.
As we follow the derived series, the first layer of fields that we construct are the possible maximal abelian subextensions.
In Table~\ref{tab:maxab64}, we list the number of these fields together with
the discriminant bound. We also present the proportion of those fields which are CM-fields and which have relative class number not in $\{1, 2, 4\}$
(the fields for which we can stop the search), as well as in the last column
the number of candidates after the sieving with the relative class numbers.
Note that by using the sieving, we can reduce the number of possible maximal abelian subfields from 10417 to 3047.

\begin{table}[h]
    \centering
    \caption{Possible maximal abelian subfields in degree 64}
    \begin{tabular}{|c|c|c|c|c|}\hline
      $\Gal(A/\Q)$ & $b : \disc(A) \leq b$ & $\#\{A\}$ & $\#\{A : A\text{ CM-field, } h_A^- \not\in \{1, 2, 4\}\}$ & $\#\{A \text{ useful}\}$\rule[-0.9ex]{0pt}{0pt} \\ \hline
				$\CC_4 \times \CC_4$ & $10^{\frac{115}{4}}$ & 49 & 32 & 17\rule{0pt}{2.6ex} \\
				$\CC_4 \times \CC_8$ & $10^{\frac{115}{2}}$ & 5 & 4 & 1 \\
				$\CC_2^5$						 & $10^{\frac{115}{2}}$ & 4 & 4 & 0 \\
				$\CC_2^3 \times \CC_4$ & $10^{\frac{115}{2}}$ & 60 & 58 & 2 \\
				$\CC_2^4$							 & $10^{\frac{115}{4}}$ & 231 & 221 & 10 \\
				$\CC_2 \times \CC_{16}$ & $10^{\frac{115}{2}}$ & 27 & 22 & 5 \\
				$\CC_2 \times \CC_4$ & $10^{\frac{115}{8}}$ & 1399 & 912 & 487 \\
				$\CC_2 \times \CC_2$ & $10^{\frac{115}{16}}$ & 5460 & 3438 & 2022 \\
				$\CC_2 \times \CC_8$ & $10^{\frac{115}{4}}$ & 92 & 70 & 22 \\
				$\CC_2^2 \times \CC_4$ & $10^{\frac{115}{4}}$ & 642 & 557 & 85 \\
				$\CC_2^3$ & $10^{\frac{115}{8}}$ & 2403 & 2009 & 394 \\
				$\CC_2^2 \times \CC_8$ & $10^{\frac{115}{2}}$ & 25 & 24 & 1 \\
				$\CC_2 \times \CC_4^2$ & $10^{\frac{115}{2}}$ & 20 & 19 & 1 \\
				\hline
    \end{tabular} \label{tab:maxab64}
\end{table}

Using these abelian extensions as a first step in the field construction we have constructed all normal CM-fields of degree 64 with maximal abelian subextension equal to one of the fields of Table~\ref{tab:maxab64}.
There are in total 81 such fields. The result is summarized in Table~\ref{tab:cmdeg64}. We list the Galois group $G$, the identifier of the group in the table of small groups~\cite{Besche2002}, the number of CM-fields found as well as the minimal discriminant found.

\begin{table}[htbp]
  \caption{Possible CM-fields of degree $64$ with relative class number one and discriminant bounded by $10^{115}$.}
  \begin{tabular}{|c|c|c|c|}\hline
    $\Gal(L/\Q)$ & ID & $\lvert\{L\}\rvert$ & $\min_L \lvert\disc(L)\rvert$ \\ \hline
    $\CC_8\rtimes_2\CC_8$ & (64, 15) & 1 & $2^{286} 3^{56}$\\ %
    $\CC_{16}\rtimes \CC_4$ & (64, 28) & 1 & $2^{232} 5^{60}$\\ %
    $\CC_2^2\rtimes \CC_{16}$ & (64, 29) & 1 & $2^{128} 17^{60}$\\ %
    $\CC_4\times \CC_2^2\rtimes \CC_4$ & (64, 58) & 1 & $2^{216} 3^{32} 5^{48}$\\ %
    $\CC_2^3._* \QQQ_8$ & (64, 66) & 2 & $2^{224} 3^{48} 5^{32}$\\ %
    $\CC_2^4.\CC_2^2$ & (64, 69) & 6 & $2^{192} 3^{48} 5^{48}$\\ %
    $\CC_2\times (\QQQ_8\rtimes^* \CC_4)$ & (64, 96) & 1 & $2^{128} 3^{48} 13^{48}$\\ %
    $\CC_2^3._* \DD_4$ & (64, 97) & 1 & $2^{64} 3^{32} 5^{48} 29^{32}$\\ %
    $\CC_2\times \CC_4\wr \CC_2$ & (64, 101) & 4 & $2^{128} 5^{48} 11^{32}$\\ %
    $\CC_2\times \CC_4.\QQQ_8$ & (64, 106) & 1 & $2^{252} 3^{56}$\\ %
    $\CC_2\times \CC_8.\CC_4$ & (64, 110) & 2 & $2^{268} 3^{56}$\\ %
    $\CC_4^2._*\CC_4$ & (64, 113) & 2 & $2^{168} 3^{48} 5^{56}$\\ %
    $\CC_8\rtimes^*\DD_4$ & (64, 116) & 5 & $2^{192} 3^{32} 5^{56}$\\ %
    $\CC_2^2\times \CC_4\rtimes \CC_4$ & (64, 194) & 1 & $2^{192} 3^{48} 5^{48}$\\ %
    $\CC_2\times \CC_4^2\rtimes \CC_2$ & (64, 195) & 1 & $2^{160} 3^{48} 5^{48}$\\ %
    $\CC_2\times \CC_4\times \DD_4$ & (64, 196) & 3 & $2^{176} 3^{32} 5^{48}$\\ %
    $\CC_2\times \CC_2^2\wr \CC_2$ & (64, 202) & 4 & $2^{176} 3^{48} 11^{32}$\\ %
    $\CC_2\times \CC_4\rtimes \DD_4$ & (64, 203) & 3 & $2^{192} 3^{48} 7^{32}$\\ %
    $\CC_2\times \CC_2^2\rtimes Q_8$ & (64, 204) & 4 & $2^{216} 3^{48} 5^{32}$\\ %
    $\CC_2\times \CC_2^2.\DD_4$ & (64, 205) & 2 & $2^{192} 3^{48} 5^{32}$\\ %
    $\CC_2\times \CC_4._*\DD_4$ & (64, 207) & 1 & $2^{192} 3^{48} 5^{48}$\\ %
    $\CC_2\times \CC_4^2\rtimes^*\CC_2$ & (64, 209) & 1 & $2^{192} 3^{48} 5^{48}$\\ %
    $\CC_2^2\times (\CC_4 . \CC_4)$ & (64, 247) & 2 & $2^{160} 3^{32} 5^{56}$\\ %
    $\CC_2\times \CC_8\circ \DD_4$ & (64, 248) & 2 & $2^{176} 3^{32} 5^{56}$\\ %
    $\CC_2\times \CC_4\circ \DD_8$ & (64, 253) & 2 & $2^{176} 3^{56} 11^{32}$\\ %
			\hline
  \end{tabular}
    \label{tab:cmdeg64}
\end{table}

Finally, for each field, we have exhibited a CM-subfield with relative class number not in $\{1, 2, 4\}$, which proves
that the relative class number of every field is non-trivial. Thus we have proven the following result.

\begin{theorem}\label{thm:64}
  Assuming GRH, there is no normal CM-field of degree 64 with relative class number one.
\end{theorem}

\section{Degree 96}\label{sec:96}

We now focus on the case of normal non-abelian CM-fields of degree 96. According
to~\cite{Besche2002} there are 224 non-abelian groups of order 96.
If $L$ is a normal CM-field of degree 96 with $h_L^- = 1$, then according to
\cite[Theorem 1]{Lee2006} its root discriminant satisfies $\rho(L) \leq 50.71$,
which implies $\disc(L) \leq 10^{164}$.
In this case, we follow a different approach: Instead of using the derived series, we list first the totally real normal fields of degree $48$ and then try to extend them to CM-fields of degree $96$, as discussed in Section~\ref{subsection:choice_series}.
The reason for this choice is that the totally real fields of degree $48$ and discriminant bounded by $10^{82}$ are only a few. This gives strong restrictions on the possible Galois groups of the CM-fields. The same was not true for the degree $64$ fields, and this explains the choice of a different strategy.
We have computed all normal totally real number fields of degree 48 with discriminant bounded by $10^{82}$. There are in total 6 fields with Galois groups as described in Table~\ref{tab:tr96}.
\begin{table}[htbp]
  \caption{Possible totally real subfields $K$ of normal CM-fields of degree 96 with relative class number one and discriminant bounded by $10^{164}$.}
  \begin{tabular}{|c|c|c|c|c|}\hline
    $\Gal(K/\Q)$ & ID & $\lvert\{K\}\rvert$ & {discriminant(s)}  \\ \hline
    $\CC_3\times \CC_4\rtimes \CC_4$ & (48, 22) & 1 & $3^{36} \cdot  5^{36} \cdot 7^{44}$ \\
    $\CC_3 \times (\CC_4 . \CC_4)$ & (48, 24) & 1 & $5^{42} \cdot 13^{46}$ \\
    $\GL_2(\F_3)$ & (48, 29) & 2 & $2^{64} \cdot 389^{24}$, $3^{42} \cdot 367^{24}$ \\
    $\CC_4 . \mathfrak A_4$ & (48, 33) & 1 &  $7^{32} \cdot 181^{24}$ \\
    $\DD_4\circ \Dic_3$ & (48, 39) & 1 & $2^{88} \cdot 5^{40} \cdot  13^{24}$ \\
    $\CC_6\times \DD_4$ & (48, 45) & 1 & $2^{108} \cdot 3^{24} \cdot 7^{44}$\\
    \hline
    \end{tabular}\label{tab:tr96}
  \end{table}

For each totally real field found we have computed the totally complex quadratic extensions, which are normal over $\Q$. There are in total 4 normal CM-fields of degree 96 and discriminant bounded by $10^{164}$, with Galois groups as described in Table~\ref{tab:cm96}.
\begin{table}[htbp]
  \caption{Possible normal CM-fields of degree 96 with relative class number one and discriminant bounded by $10^{164}$.}
  \begin{tabular}{|c|c|c|c|c|}\hline
    $\Gal(L/\Q)$ & ID & $\#\{L\}$ & {discriminant}  \\ \hline
    $\CC_3\times \CC_8\rtimes \CC_4$ & (96, 47) & 1 & $5^{84} \cdot 13^{92}$ \\
    $\CC_6\times \CC_4\rtimes \CC_4$ & (96, 163) & 1 & $3^{72} \cdot  5^{72} \cdot 7^{88}$\\
    $\CC_4\rtimes \mathfrak S_4$ & (96, 187) & 1 & $3^{84} \cdot 367^{48}$ \\
    $\DD_4\times \CC_2\times \CC_6$ & (96, 221) & 1 & $2^{216} \cdot 3^{48} \cdot 7^{88}$\\
    \hline
    \end{tabular}\label{tab:cm96}
  \end{table}
For each of these fields, we have exhibited a CM-subfield of relative class number $> 4$. 
Thus using~\cite[Theorem 1]{Okazaki2000} we have proven the following.

\begin{theorem}\label{thm:96}
Assuming GRH, there is no normal CM-field of degree 96 with (relative) class number one.
\end{theorem}

\section{Relative class number one problem for normal CM-fields}\label{sec:relcmone}

While the results from the preceding two sections solve the class number one problem for normal CM-fields, they are not sufficient to solve
the equivalent problem for relative class numbers.
Although the previous results of various authors do often include the classification of CM-fields with relative class number one, they do not cover all the possible degrees and Galois groups.
The purpose of this section is to complete this classification by computing the remaining cases.

\subsection{Degree bounds}
We first recall the bounds on the degree of a normal CM-field of relative class number one.
Denote by
\[ 
    d = \limsup_L \{ [L : \Q]\}
\]
where $L$ runs through all normal CM-fields with relative class number one.
Define $d_{\, \mathrm{GRH}}$ in a similar way assuming GRH.
It was shown by Odlyzko~\cite{Odlyzko1975} that $d < \infty$.
An explicit bound was provided by Hoffstein~\cite{Hoffstein1979}, who showed that $d \leq 434$.
These bounds were subsequently improved to $d \leq 266$ and $d_{\, \mathrm{GRH}} \leq 164$ by Bessassi~\cite{Bessassi2003}.
The best currently known bounds were later obtained by Lee--Kwon~\cite{Lee2006}, who proved that $d \leq 216$ and $d_{\, \mathrm{GRH}} \leq 96$.

\subsection{Abelian fields}

The abelian CM-fields with class number one were determined by Yamamura in~\cite{Yamamura1994}. There are 172 such fields, the largest field having degree $24$.
The corresponding relative class number one problem was solved by Chang--Kwon in~\cite{Chang2000}. There are 300 such fields, the largest field having degree $24$
(see Table~\ref{tab:allthefieldsdegrees}.)

\subsection{Non-abelian fields}
We now focus on the case of non-abelian fields. Under GRH, we know that the degree of a normal CM-field with relative class number one is bounded by $96$. The following theorems due to various authors give further constraints.

\begin{theorem}\label{thm:structure_cnp}
  Let $L$ be a normal CM-field of degree $d$ with relative class number one and Galois group $G$. Assume that $G$ is not abelian.
  Then the following hold:
  \hfill
  \begin{enumerate}
    \item If $d$ is equal to $4p^2$ for an odd prime $p$, then $p = 3$ and $G = \DD_{6} \times \CC_3$.
    \item If $d$ is equal to $4p$ for an odd prime $p$, then $G$ is dihedral.
    \item The degree satisfies $d \neq 2pq$ for all odd distinct primes $p, q$.
    \item The group $G$ is not dicyclic.
\end{enumerate}
\end{theorem}

\begin{proof}
 It follows from {\cite[Theorem 1]{Chang2002}}, {\cite[Theorem 1]{Kwon2009}}, 
{\cite[Theorem 7, Lemma 3]{Louboutin1997}} and~\cite[Theorem 7]{Louboutin1999}.
\end{proof}
\begin{theorem}\label{thm:bound_degree_dihedral}
  Let $L$ be a normal CM-field with Galois group $\DD_{2^rl}$, $l$ odd, and relative class number one.
  Then $r \leq 4$ and $l \leq 7$.
\end{theorem}
\begin{proof}
  Follows directly from \cite[Théorème 1.1]{Lefeuvre2000}.
\end{proof}

As a consequence, we have the following:
\begin{corollary}
  Let $L$ be a normal CM-field of degree $d$ with relative class number one and Galois group $G$. Assume that $G$ is not abelian.
  Then 
  \[ 
    d \in \{8, 12, 16, 20, 24, 32, 36, 40, 48, 54, 56, 60, 64,  72,  80, 84, 88, 96 \}.
  \]
\end{corollary}
\begin{proof}
  By Theorem~\ref{thm:structure_cnp}, we know that $d$ must be different from $2pq$ for odd distinct primes $p, q$.
  Moreover, as the center of $G$ must have even order, we also know that $d$ must be different from $2p$ where $p$ is a prime number. With these observations, we get
  \[ 
    d \in \{8, 12, 16, 20, 24, 28, 32, 36, 40, 44, 48, 52, 54, 56, 60, 64, 68, 72, 76, 80, 84, 88, 92, 96 \}.
  \]
  Now, we notice that for $d\in \{ 28, 44, 52, 68, 76, 92\}$, we have by Theorem~\ref{thm:structure_cnp} that the only possible Galois groups are the dihedral groups and by Theorem~\ref{thm:bound_degree_dihedral} there are no CM-fields with relative class number one of these degrees.
\end{proof}

We will now treat every possible degree and complete the classification.
 We will make use of the 
upper bounds on (root) discriminants given in Table~\ref{tab:discbound}, as proven in~\cite[Theorem 1]{Lee2006}.
Since the corresponding table~\cite[Table 6]{Lee2006} does not include all the degrees we need to consider,
we have reimplemented the methods of~\cite{Lee2006} and determined the missing values.
For every degree $d$ of interest, Table~\ref{tab:discbound} contains values $\alpha(d)$ and $B(d)$, such that for
any normal CM-field $L$ of degree $d$ with relative class number one the root discriminant satisfies $\rho(L) \leq \alpha(d)$ and the discriminant satisfies $\lvert \disc(L) \rvert \leq B(d)$.
In the following we recall results of other authors and complement them with explicit computations. The results of our computations are presented in propositions.
The methods used are the ones from Section~\ref{sec:construction} together with the explicit bounds from Table~\ref{tab:discbound}.

\begin{table}[htbp]
\centering
\caption{(Root) discriminant bounds for normal CM-fields of degree $d$ and relative class number one.}
\begin{tabular}{|c|r|c|}\hline
  $d$ &\multicolumn{1}{c|}{$\alpha(d)$} & $B(d)$ \\ \hline
 $16$ & $324.70$ & $10^{41}$ \\
 $24$ & $162.10$ & $10^{54}$ \\
 $32$ & $109.30$ & $10^{66}$ \\
 $40$ & $87.36$ & $10^{78}$ \\
 $48$ & $75.08$ & $10^{91}$ \\
 $56$ & $69.84$ & $10^{103}$ \\
 $72$ & $60.19$ & $10^{129}$ \\
 $80$ & $54.98$ & $10^{140}$ \\
 $88$ & $54.61$ & $10^{152}$ \\
 \hline
 \end{tabular}
\label{tab:discbound}
 \end{table}

\subsubsection*{Case $d =8$}

In this case $G$ is either dihedral or quaternion.
The class number one problem was solved by Louboutin--Okazaki in~\cite{Louboutin1994}, who showed that there are 17 such fields.
As part of their result, they showed that if $G = \QQQ_8$, then $h_L^- > 1$ (\cite[Theorem 1]{Louboutin1994}).
The case of the dihedral group is covered by \cite[Thérorème 1.1]{Lefeuvre2000}: there are 19 CM-fields of degree $8$ and Galois group $\DD_4$ with relative class number one, of which 17 have class number one.

\subsubsection*{Case $d =12$}

In this case, $G = \DD_{6}$ by Theorem~\ref{thm:structure_cnp} and the fields are determined in~\cite[Theorem 1]{Louboutin1994}.
There are $16$ fields with relative class number one, of which $9$ have class number one.
The fields are listed in \cite[Table 1]{Louboutin1994} as compositions
of biquadratic fields with totally real non-normal cyclic fields. The second to
last entry is $L = \Q(\sqrt{-7}, \sqrt{-35}, \alpha)$, where $\alpha$ has
minimal polynomial $X^3 - 27X^2 - 51$. Note that there is a typo since the field $L$ is neither normal nor
has relative class number one.  The correct minimal polynomial for the totally
real cyclic field must be $X^3 - 27X^2 + 51$.

\subsubsection*{Case $d =16$}

In~\cite[Theorem 1]{Louboutin1997c}, it is shown that exactly one of the following three statements holds true:
\begin{enumerate}
\item
  The Galois group $G$ is dihedral.
\item
  The field $L = L_1 L_2$ is the composition of two normal octic CM-fields $L_1$ and $L_2$ with the same maximal real subfield.
\item
  The Galois group $G$ is the Pauli group $\CC_4 \circ \DD_4$. %
\end{enumerate}

Moreover, the authors solve the relative class number one problem for case~(1), case~(3) and case~(2) if $L_1$ or $L_2$ is a quaternion field. In total 7 fields with relative class number one are found, of which 6 have class number one.

Some subproblems of case~(2) are treated in~\cite{Louboutin1997b}.
If $L_1$ and $L_2$ are dihedral, then it is shown that there are exactly 2 such
fields with relative class number one and both of these fields have class number one (\cite[Theorem 2]{Louboutin1997b}).
If $L_1$ is abelian and $L_2$ is dihedral, then only the class number problem one is solved and it is shown that there are exactly 4 such fields with class number one (\cite[Theorem 3]{Louboutin1997b}).

Thus it remains consider the relative class number one problem for fields $L = L_1 L_2$, where $L_1$ is abelian and $L_2$ is dihedral.
Note that is easily checked that this forces the center of $G$ to be non-cyclic and $G$ to be isomorphic to $\CC_2 \times \DD_4$, $\CC_4\rtimes \CC_4$ or $\CC_2^2 \rtimes \CC_4$.
\begin{proposition}
  \label{prop:missing:16}
      There are exactly 4 normal CM-fields $L$ with relative class number one and 
      \[ G \in \{\CC_2 \times \DD_4, \CC_4 \rtimes \CC_4, \CC_2^2 \rtimes \CC_4\}.\]
\end{proposition}
In total, there are $13$ fields with relative class number one of which 12 have class number one.

\subsubsection*{Case $d =20$}
By Theorem~\ref{thm:structure_cnp} we have $G = \DD_{10}$. This case is covered by \cite[Thérorème 1.1]{Lefeuvre2000}: There are 2 fields with relative class number one, of which one has class number one.

\subsubsection*{Case $d =24$}

Although the class number one problem is solved, only partial information is known about the relative class number one problem.
To identify the missing cases, we recall the following summary of \cite[Introduction]{Park2002}.
The possible Galois groups for non-abelian normal CM-fields of degree $24$ are
$\CC_2 \times \Alt_4$, $\SL_2(\F_3)$, $C_3 \times \QQQ_8$, $\QQQ_{24}$, $\CC_3 \rtimes \CC_8$, $\DD_{12}$, $\CC_2 \times \QQQ_{12}$, $\CC_2 \times \D_{12}$, $\CC_4 \times \DD_3$, $\CC_3 \times \DD_4$ and $\CC_3 \rtimes \DD_4$.

\begin{itemize}
  \item
    If $G = \CC_2 \times \Alt_4$ then there are exactly two fields with relative
    class number one and both fields have class number one (\cite[Theorem 14]{Lemmermeyer1999}).
  \item
    If $G = \SL_2(\F_3)$, then there is exactly one field with class number one (\cite[Theorem 18]{Lemmermeyer1999}).
  \item
    If $G \in \{\CC_3 \times \QQQ_8\, \QQQ_{24}, \CC_3 \rtimes \CC_8\}$, then there is no field with relative class number one (\cite[Corollary 4 \& Corollary 6]{Louboutin1999}).
  \item
    If $G = \DD_{12}$, then this is covered by \cite[Thérorème 1.1]{Lefeuvre2000}. There is exactly one field with (relative) class number one.
  \item
    If $G = \CC_2 \times \DD_{6}$ then there is exactly one field with class number one (\cite[Theorem 1~(1)]{Park2002}).
  \item
    If $G = \CC_4 \times \DD_3$, then there is exactly one field with relative class number one and this field has class number one (\cite[Theorem 1~(2)]{Park2002}).
  \item
    If $G = \CC_3 \times \DD_4$, then there is exactly one field with relative class number one and this field has class number one (\cite[Theorem 1~(3)]{Park2002}).
.
  \item
    If $G = \{ \CC_2 \times \QQQ_{12}, \CC_3 \rtimes \DD_4\}$, then there is no field with relative class number one (\cite[Theorem 1~(4)]{Park2002}).
\end{itemize}

Thus there are exactly 7 fields with class number one and at least 7 fields with relative class number one, with possibly more fields for the groups $\SL_2(\F_3)$ and $\CC_2 \times \DD_{6}$.
\begin{proposition}
  \label{prop:missing:24_1}
      There is exactly one CM-field with Galois group $\SL_2(\F_3)$ and relative class number one. This field has class number one.
\end{proposition}

\begin{proposition}
  \label{prop:missing:24_2}
      There is exactly one CM-field with Galois group $\CC_2 \times \DD_{6}$ and relative class number one. This field has class number one.
\end{proposition}

It follows that there are exactly 7 fields with relative class number one.

\subsubsection*{Case $d =32$}

\begin{proposition}
  \label{prop:missing:32}
      There are exactly 4 CM-fields of degree 32 with relative class number one. All of these fields have class number one.
\end{proposition}

\begin{remark}\label{rem:32}
  The normal CM-fields of degree $32$ with class number one had already been determined by Park--Yang--Kwon in~\cite{Park2007}.
  They show that
  \begin{itemize}
    \item
      there is exactly one field with class number one which is not the compositum of two normal CM-subfields of degree $16$ with the same maximal real subfield (\cite[Theorem 1~(1)]{Park2007}),
    \item
      there are exactly five fields with relative class number one which are
      composita of two normal CM-subfields of degree $16$ with the same maximal
      real subfield and all five fields have class number one
      (\cite[Theorem 1~(2)]{Park2007}).
      The five fields are explicitly given as
      \begin{align*} &\Q(\sqrt{2 + \sqrt{2}}, \sqrt{3 + \sqrt{3}}, \sqrt{-1}),\\
                     &\Q(\sqrt{2}, \sqrt{17 + 4 \sqrt{17}}, \sqrt{-(5 + \sqrt{17})/2}, \sqrt{-(17 + 3 \sqrt{17})/2}, \alpha),\\
&\Q(\theta, \sqrt 2, \sqrt{-1}), \Q(\theta, \sqrt 3, \sqrt{-1}), \Q(\theta, \sqrt 7, \sqrt{-1}),
\end{align*}
where $\alpha^8 + 17\alpha^7 + 85\alpha^4 + 136\alpha^2 + 68 = 0$ and $\theta^8 + 10\theta^6 + 25\theta^4 + 20\theta^2 + 5 = 0$.
  \end{itemize}
  Thus, according to~\cite{Park2007} there are in total 6 fields with class number one, whereas we have shown that there are only 4 fields with class number one.
  This discrepancy can be solved by noticing that the two fields $\Q(\theta, \sqrt 3, \sqrt{-1})$ and $\Q(\theta, \sqrt 7, \sqrt{-1})$ of the second case are in fact neither normal nor they have class number one.
\end{remark}

\subsubsection*{Case $d =36$} In~\cite[Theorem 1]{Chang2002} it is shown that there are exactly three fields with relative class number one.
All of them have Galois group $\DD_{6} \times \CC_3$ and class number one.

\subsubsection*{Case $d =40$} In~\cite[Theorem 2]{Park2002} it is shown that there exists exactly one field with class number one.
To finish the classification of the fields with relative class number one, note that in~\cite{Park2002} it is shown that the Galois group of any such field satisfies
\[ G \in \{ \CC_2 \times \Dic_5, \CC_2 \times \DD_{10}, \CC_5 \rtimes \DD_4, \CC_4 \rtimes \DD_{5}, \CC_5 \times \DD_4, \CC_2 \times (\F_5 \rtimes \F_5^\times) \}. \] 
\begin{proposition}
  \label{prop:missing:40}
      There is exactly one CM-field with relative class number one and 
      \[ G \in \{\CC_2 \times \Dic_5, \CC_2 \times \DD_{10}, \CC_5 \rtimes \DD_4, \CC_4 \rtimes \DD_{5}, \CC_5 \times \DD_4, \CC_2 \times (\F_5 \rtimes \F_5^\times)\}. \] 
\end{proposition}

Thus it follows that the field found in \cite{Park2002} is the only field with relative class number one.

\subsubsection*{Case $d =48$}
In \cite[Theorem 1]{Chang2002} it is shown that there exists exactly one normal CM-field with class number one which has a normal CM-subfield of degree 16. In~\cite[Proposition 4.1]{Kwon2007} it is shown that there is no other normal CM-field of degree $48$ with class number one.

\begin{proposition}
  \label{prop:missing:48}
      There is exactly one CM-field of degree $48$ with relative class number one. This field has class number one.
\end{proposition}

It follows that there is exactly one normal CM-field with relative class number one.

\subsubsection*{Case $d \in \{54, 60, 84\}$}
In \cite[Proposition 7.1, Proposition 6.1]{Kwon2007} it is shown that there is no CM-field with relative class number one.

\subsubsection*{Case $d =56$}
It follows from \cite[Section 3]{Kwon2007} that a CM-field $L$ of degree 56 has a unique normal CM-subfield $L_8$ of degree $8$.
Moreover, according to ~\cite[Proposition 3.1]{Kwon2007}, we have $h_N > 1$ and $h_N^- > 1$ whenever $\Gal(L_8/\Q) \not\cong \CC_2^3$.
As $\Gal(L_8/\Q) \cong \CC_2^3$ implies $G \cong \CC_2^2 \times \DD_7$ it follows
from the following computation that there is no CM-field with relative class
number one.
\begin{proposition}
  \label{prop:missing:56}
      There is no normal CM-field with relative class number one and Galois group $\CC_2^2 \times \DD_7$.
\end{proposition}

\subsubsection*{Case $d \in \{64, 96\}$}
We have shown in Theorem~\ref{thm:64} and Theorem~\ref{thm:96} that there is no such field.

\subsubsection*{Case $d \in \{72, 80\}$}
In \cite[Proposition 5.1, Proposition 4.2]{Kwon2007} it is shown that there is no CM-field with class number one.
\begin{proposition}
  \label{prop:missing:72}
      There is no CM-field with relative class number one of degree $72$.
\end{proposition}

\begin{proposition}
        \label{prop:missing:80}
      There is no CM-field with relative class number one of degree $80$.
\end{proposition}

It follows that there is no CM-field with relative class number one.

\subsubsection*{Case $d =88$}
It follows from \cite[Section 3]{Kwon2007} that a normal CM-field $L$ of degree 88 has a unique normal CM-subfield $L_8$ of degree $8$.
Moreover according to~\cite[Proposition 3.1]{Kwon2007} we have $h_N^- > 1$ if $\Gal(L_8/\Q) \not\cong \CC_2^3$.
As $\Gal(L_8/\Q) \cong \CC_2^3$ implies $G \cong \CC_2^2 \times \DD_{11}$ it follows
from the following proposition that there is no CM-field with relative class
number one.
\begin{proposition}
        \label{prop:missing:88}
      There is no CM-field with relative class number one and Galois group $\CC_2^2 \times \DD_{11}$.
\end{proposition}

\subsection{Summary}

By putting the results of the previous paragraphs together, we obtain:

\begin{theorem}
  Assuming GRH, there are exactly 368 normal CM-fields with relative class number one and 227 normal CM-fields with class number one.
  The number of these CM-fields for a given degree and Galois group are given in Tables~\ref{tab:allthefieldsdegrees} and~\ref{tab:allthefieldsgalois}.
  The fields themselves are listed in Appendix~\ref{appendix:fields}.
\end{theorem}

\section{A note on the computations}\label{sec:comp}

The main algorithmic tools for constructing relative abelian extension of number fields had already been implemented as part of~\cite{Fieker2019} in the system
\textsc{Hecke}~\cite{Hecke}.
For dealing with the various algorithmic group theoretic questions we used \textsc{GAP}~\cite{GAP4}. %
In this regard, the table of small groups as determined in~\cite{Besche2002} has been essential.
The final computation of (relative) class numbers (of subfields) have additionally been checked with \textsc{Magma}~\cite{Magma}.

In Table~\ref{tab:runtime} we give an overview of the runtime of the computation.

\begin{table}[htpb]
  \caption{Runtime}
  \label{tab:runtime}
  \begin{tabular}{|c|c|c|}
    \hline
    \multicolumn{2}{|c|}{Problem} & Runtime \\\hline
      \multirow{4}{*}{Degree 64}& Computation of Table~\ref{tab:maxab64} & 1h \\
                                & Computation of Table~\ref{tab:cmdeg64} & 160d \\
                                & Proof $h_L^{-} > 1$  & 2d \\\hline
       \multirow{3}{*}{Degree 96} & Computation of Table~\ref{tab:tr96} & 1h \\
                                  & Computation of Table~\ref{tab:cm96} & 1h \\
                                  & Proof $h_L^{-} > 1$  & 2h \\
                                 \hline
       \multirow{1}{*}{Other degrees}& Section~\ref{sec:relcmone} & 120d \\

    \hline
       \multicolumn{2}{|c|}{Total time} & 284d \\\hline
  \end{tabular}
\end{table}

 \newpage
 \appendix

 \section{The CM-fields with (relative) class number one}\label{appendix:fields}

 We provide a list of all normal CM-fields with (relative) class number one (assuming GRH).\footnote{The table of fields is also available on the homepage of the first author.}
 Note that our results imply that these fields are in principle already available from the works of other authors.
 On the other hand, in lots of cases these fields were not given as explicit as possible.
 For example, they were given as ray class fields or Hilbert class fields of other fields, or in the abelian case, the fields were specified using Dirichlet characters.

 For the readers convenience, in our table each field is specified by the minimal
 polynomial of a primitive element.
 These minimal polynomials have been obtained in the following way.
 For abelian groups, we have (re)computed all CM-fields with conductor at most 65689 and degree bounded by 24, which by~\cite[Theorem 1]{Chang2000} include all abelian CM-fields with (relative) class number one.
 For the fields of degree $\geq 10$, we have again used~\cite{Lee2006} to obtain explicit discriminant bounds for normal CM-fields with relative class number one. Using the same methods that we used in degree 64 and degree 96, we have then constructed all normal CM-fields with relative class number one.
 In the course of these calculations, with two exceptions, we could confirm previous results from the literature. In case of $\DD_6$ we found a typo in~\cite[Table 1]{Louboutin1994} and in case of degree 32 we corrected the results from~\cite{Park2007} (see Remark~\ref{rem:32}).
 The only missing Galois group is $\DD_4$, for which we took the polynomials from~\cite{Louboutin1994}.

\clearpage{}%
\renewcommand*{\arraystretch}{1.1}
\begin{landscape}
 \begin{center}
 \footnotesize

\begin{tabular}{|>{\centering\arraybackslash}p{0.2cm}|c|>{\centering\arraybackslash}p{7.4cm}|c|>{\centering\arraybackslash}p{1.6cm}||>{\centering\arraybackslash}p{0.2cm}|c|>{\centering\arraybackslash}p{7.4cm}|c|>{\centering\arraybackslash}p{1.6cm}|}\hline
   $d$ & $G $ & $f$ & $h_L$ & $ \disc(L)$ & $d$ & $ G$ & $f$ & $h_L$ & $ \disc(L)$ \\\hline

2 & $\CC_2$ & \begin{minipage}{7.4cm}\begin{center}$\begin{aligned}x^{2} - x + 1\end{aligned}$\end{center}\end{minipage} & 1 & $3$ & 
4 & $\CC_2^2$ & \begin{minipage}{7.4cm}\begin{center}$\begin{aligned}x^{4} - 5x^{2} + 25\end{aligned}$\end{center}\end{minipage} & 2 & $2^{4}3^{2}5^{2}$\\
  &           & \begin{minipage}{7.4cm}\begin{center}$\begin{aligned}x^{2} + 1\end{aligned}$\end{center}\end{minipage} & 1 & $2^{2}$ & 
  &           & \begin{minipage}{7.4cm}\begin{center}$\begin{aligned}x^{4} - 9x^{2} + 25\end{aligned}$\end{center}\end{minipage} & 1 & $2^{4}19^{2}$\\
  &           & \begin{minipage}{7.4cm}\begin{center}$\begin{aligned}x^{2} - x + 2\end{aligned}$\end{center}\end{minipage} & 1 & $7$ & 
  &           & \begin{minipage}{7.4cm}\begin{center}$\begin{aligned}x^{4} + 9x^{2} + 1\end{aligned}$\end{center}\end{minipage} & 1 & $7^{2}11^{2}$\\
  &           & \begin{minipage}{7.4cm}\begin{center}$\begin{aligned}x^{2} + 2\end{aligned}$\end{center}\end{minipage} & 1 & $2^{3}$ & 
  &           & \begin{minipage}{7.4cm}\begin{center}$\begin{aligned}x^{4} - 4x^{2} + 9\end{aligned}$\end{center}\end{minipage} & 2 & $2^{8}5^{2}$\\
  &           & \begin{minipage}{7.4cm}\begin{center}$\begin{aligned}x^{2} - x + 3\end{aligned}$\end{center}\end{minipage} & 1 & $11$ & 
  &           & \begin{minipage}{7.4cm}\begin{center}$\begin{aligned}x^{4} + 25\end{aligned}$\end{center}\end{minipage} & 2 & $2^{8}5^{2}$\\
\hline\end{tabular}
\newpage

\begin{tabular}{|>{\centering\arraybackslash}p{0.2cm}|c|>{\centering\arraybackslash}p{7.4cm}|c|>{\centering\arraybackslash}p{1.6cm}||>{\centering\arraybackslash}p{0.2cm}|c|>{\centering\arraybackslash}p{7.4cm}|c|>{\centering\arraybackslash}p{1.6cm}|}\hline
   $d$ & $G $ & $f$ & $h_L$ & $ \disc(L)$ & $d$ & $ G$ & $f$ & $h_L$ & $ \disc(L)$ \\\hline

4 & $\CC_2^2$ & \begin{minipage}{7.4cm}\begin{center}$\begin{aligned}x^{4} - x^{3} - 28x^{2} - 29x + 841\end{aligned}$\end{center}\end{minipage} & 2 & $3^{2}5^{2}23^{2}$ & 
4 & $\CC_2^2$ & \begin{minipage}{7.4cm}\begin{center}$\begin{aligned}x^{4} + 43x^{2} + 144\end{aligned}$\end{center}\end{minipage} & 1 & $19^{2}67^{2}$\\
  &           & \begin{minipage}{7.4cm}\begin{center}$\begin{aligned}x^{4} + 29x^{2} + 121\end{aligned}$\end{center}\end{minipage} & 2 & $3^{2}7^{2}17^{2}$ & 
  &           & \begin{minipage}{7.4cm}\begin{center}$\begin{aligned}x^{4} - 2x^{3} + 87x^{2} - 86x + 1523\end{aligned}$\end{center}\end{minipage} & 3 & $2^{6}163^{2}$\\
  &           & \begin{minipage}{7.4cm}\begin{center}$\begin{aligned}x^{4} - 45x^{2} + 529\end{aligned}$\end{center}\end{minipage} & 2 & $2^{4}7^{2}13^{2}$ & 
  &           & \begin{minipage}{7.4cm}\begin{center}$\begin{aligned}x^{4} + 97x^{2} + 2025\end{aligned}$\end{center}\end{minipage} & 2 & $7^{2}11^{2}17^{2}$\\
  &           & \begin{minipage}{7.4cm}\begin{center}$\begin{aligned}x^{4} - 2x^{3} + 31x^{2} - 30x + 134\end{aligned}$\end{center}\end{minipage} & 2 & $2^{4}7^{2}13^{2}$ & 
  &           & \begin{minipage}{7.4cm}\begin{center}$\begin{aligned}x^{4} - 2x^{3} + 45x^{2} - 44x + 149\end{aligned}$\end{center}\end{minipage} & 2 & $2^{4}5^{2}67^{2}$\\
  &           & \begin{minipage}{7.4cm}\begin{center}$\begin{aligned}x^{4} - 27x^{2} + 289\end{aligned}$\end{center}\end{minipage} & 1 & $7^{2}61^{2}$ & 
  &           & \begin{minipage}{7.4cm}\begin{center}$\begin{aligned}x^{4} + 67x^{2} + 784\end{aligned}$\end{center}\end{minipage} & 2 & $3^{2}11^{2}41^{2}$\\
\hline\end{tabular}
\newpage

\begin{tabular}{|>{\centering\arraybackslash}p{0.2cm}|c|>{\centering\arraybackslash}p{7.4cm}|c|>{\centering\arraybackslash}p{1.6cm}||>{\centering\arraybackslash}p{0.2cm}|c|>{\centering\arraybackslash}p{7.4cm}|c|>{\centering\arraybackslash}p{1.6cm}|}\hline
   $d$ & $G $ & $f$ & $h_L$ & $ \disc(L)$ & $d$ & $ G$ & $f$ & $h_L$ & $ \disc(L)$ \\\hline

4 & $\CC_2^2$ & \begin{minipage}{7.4cm}\begin{center}$\begin{aligned}x^{4} + 219x^{2} + 10816\end{aligned}$\end{center}\end{minipage} & 2 & $7^{2}11^{2}61^{2}$ & 
4 & $\CC_2^2$ & \begin{minipage}{7.4cm}\begin{center}$\begin{aligned}x^{4} + 199x^{2} + 324\end{aligned}$\end{center}\end{minipage} & 2 & $5^{2}47^{2}163^{2}$\\
  &           & \begin{minipage}{7.4cm}\begin{center}$\begin{aligned}x^{4} + 79x^{2} + 324\end{aligned}$\end{center}\end{minipage} & 2 & $5^{2}23^{2}43^{2}$ & 
  &           & \begin{minipage}{7.4cm}\begin{center}$\begin{aligned}x^{4} + 215x^{2} + 676\end{aligned}$\end{center}\end{minipage} & 2 & $3^{2}89^{2}163^{2}$\\
  &           & \begin{minipage}{7.4cm}\begin{center}$\begin{aligned}x^{4} + 143x^{2} + 3844\end{aligned}$\end{center}\end{minipage} & 6 & $3^{2}19^{2}89^{2}$ & 
  &           & \begin{minipage}{7.4cm}\begin{center}$\begin{aligned}x^{4} + 283x^{2} + 3600\end{aligned}$\end{center}\end{minipage} & 2 & $13^{2}31^{2}163^{2}$\\\cline{6-10}
  &           & \begin{minipage}{7.4cm}\begin{center}$\begin{aligned}x^{4} + 99x^{2} + 1024\end{aligned}$\end{center}\end{minipage} & 2 & $5^{2}7^{2}163^{2}$ & 
6 & $\CC_6$ & \begin{minipage}{7.4cm}\begin{center}$\begin{aligned}x^{6} + x^{5} + x^{4} + x^{3} + x^{2} + x + 1\end{aligned}$\end{center}\end{minipage} & 1 & $7^{5}$\\
  &           & \begin{minipage}{7.4cm}\begin{center}$\begin{aligned}x^{4} - 2x^{3} + 79x^{2} - 78x + 47\end{aligned}$\end{center}\end{minipage} & 2 & $2^{6}11^{2}67^{2}$ & 
  &           & \begin{minipage}{7.4cm}\begin{center}$\begin{aligned}x^{6} - x^{3} + 1\end{aligned}$\end{center}\end{minipage} & 1 & $3^{9}$\\
\hline\end{tabular}

\end{center}\end{landscape}
\clearpage{}%
\clearpage{}%
\renewcommand*{\arraystretch}{1.1}
\begin{landscape}
 \begin{center}
 \footnotesize
 
\begin{tabular}{|>{\centering\arraybackslash}p{0.4cm}|>{\centering\arraybackslash}p{1.9cm}|>{\centering\arraybackslash}p{17.7cm}|c|>{\centering\arraybackslash}p{1.6cm}|}\hline
   $d$ & $G $ & $f$ & $h_L$ & $ \disc(L)$ \\\hline

8 & $\CC_8$ & \begin{minipage}{17.7cm}\begin{center}$\begin{aligned} x^{8} + 8x^{6} + 20x^{4} + 16x^{2} + 2\end{aligned}$\end{center}\end{minipage} & 1 & $2^{31}$\\
  &         & \begin{minipage}{17.7cm}\begin{center}$\begin{aligned} x^{8} + x^{7} + 3x^{6} + 11x^{5} + 44x^{4} - 53x^{3} + 153x^{2} - 160x + 59\end{aligned}$\end{center}\end{minipage} & 1 & $41^{7}$\\
  \hline
8 & $\CC_2\times \CC_4$ & \begin{minipage}{17.7cm}\begin{center}$\begin{aligned} x^{8} - x^{7} + x^{5} - x^{4} + x^{3} - x + 1\end{aligned}$\end{center}\end{minipage} & 1 & $3^{4} 5^{6}$\\
  &                     & \begin{minipage}{17.7cm}\begin{center}$\begin{aligned} x^{8} - x^{6} + x^{4} - x^{2} + 1\end{aligned}$\end{center}\end{minipage} & 1 & $2^{8} 5^{6}$\\
  &                     & \begin{minipage}{17.7cm}\begin{center}$\begin{aligned} x^{8} + 1\end{aligned}$\end{center}\end{minipage} & 1 & $2^{24}$\\
  &                     & \begin{minipage}{17.7cm}\begin{center}$\begin{aligned} x^{8} - x^{7} - x^{6} + 3x^{5} - x^{4} + 6x^{3} - 4x^{2} - 8x + 16\end{aligned}$\end{center}\end{minipage} & 1 & $5^{6} 7^{4}$\\
  &                     & \begin{minipage}{17.7cm}\begin{center}$\begin{aligned} x^{8} - 2x^{6} + 4x^{4} - 8x^{2} + 16\end{aligned}$\end{center}\end{minipage} & 1 & $2^{12} 5^{6}$\\
\hline\end{tabular}
\newpage

\begin{tabular}{|>{\centering\arraybackslash}p{0.4cm}|>{\centering\arraybackslash}p{1.9cm}|>{\centering\arraybackslash}p{17.7cm}|c|>{\centering\arraybackslash}p{1.6cm}|}\hline
   $d$ & $G $ & $f$ & $h_L$ & $ \disc(L)$ \\\hline

8 & $\DD_4$ & \begin{minipage}{17.7cm}\begin{center}$\begin{aligned} x^{8} + 26x^{6} + 28x^{5} + 99x^{4} + 8x^{3} - 2x^{2} + 88x + 68\end{aligned}$\end{center}\end{minipage} & 1 & $2^{12} 89^{4}$\\
  &         & \begin{minipage}{17.7cm}\begin{center}$\begin{aligned} x^{8} + 2x^{7} + 28x^{6} + 54x^{5} + 187x^{4} + 156x^{3} - 25x^{2} - 21x + 19\end{aligned}$\end{center}\end{minipage} & 1 & $5^{4} 149^{4}$\\
  &         & \begin{minipage}{17.7cm}\begin{center}$\begin{aligned} x^{8} - 2x^{7} + 38x^{6} - 74x^{5} + 337x^{4} - 306x^{3} + 65x^{2} - 59x + 89\end{aligned}$\end{center}\end{minipage} & 3 & $5^{4} 269^{4}$\\
  &         & \begin{minipage}{17.7cm}\begin{center}$\begin{aligned} x^{8} - 2x^{7} + 21x^{6} + 10x^{5} + 250x^{4} + 120x^{3} + 488x^{2} + 5774x + 6563\end{aligned}$\end{center}\end{minipage} & 1 & $29^{4} 53^{4}$\\
  &         & \begin{minipage}{17.7cm}\begin{center}$\begin{aligned} x^{8} + 42x^{6} - 44x^{5} + 235x^{4} + 8x^{3} - 114x^{2} - 128x + 124\end{aligned}$\end{center}\end{minipage} & 1 & $2^{12} 233^{4}$\\
  &         & \begin{minipage}{17.7cm}\begin{center}$\begin{aligned} x^{8} + 45x^{6} + 417x^{4} + 180x^{2} + 16\end{aligned}$\end{center}\end{minipage} & 1 & $5^{4} 389^{4}$\\
  &         & \begin{minipage}{17.7cm}\begin{center}$\begin{aligned} x^{8} + 2x^{7} + 44x^{6} - 4x^{5} + 87x^{4} + 192x^{3} - 127x^{2} - 195x + 117\end{aligned}$\end{center}\end{minipage} & 1 & $13^{4} 157^{4}$\\
\hline\end{tabular}
\newpage

\begin{tabular}{|>{\centering\arraybackslash}p{0.4cm}|>{\centering\arraybackslash}p{1.9cm}|>{\centering\arraybackslash}p{17.7cm}|c|>{\centering\arraybackslash}p{1.6cm}|}\hline
   $d$ & $G $ & $f$ & $h_L$ & $ \disc(L)$ \\\hline

12 & $\CC_{12}$ & \begin{minipage}{17.7cm}\begin{center}$\begin{aligned} x^{12} - x^{11} + 3x^{10} - 4x^{9} + 9x^{8} + 2x^{7} + 12x^{6} + x^{5} + 25x^{4} - 11x^{3} + 5x^{2} - 2x + 1\end{aligned}$\end{center}\end{minipage} & 1 & $5^{9} 7^{8}$\\
   &            & \begin{minipage}{17.7cm}\begin{center}$\begin{aligned} x^{12} + 3x^{10} - x^{9} + 9x^{8} + 9x^{7} + 28x^{6} + 18x^{5} + 75x^{4} + 26x^{3} + 9x^{2} + 3x + 1\end{aligned}$\end{center}\end{minipage} & 1 & $3^{16} 5^{9}$\\
   &            & \begin{minipage}{17.7cm}\begin{center}$\begin{aligned} x^{12} + x^{11} + 2x^{10} - 20x^{9} - 13x^{8} - 19x^{7} + 85x^{6} + 51x^{5} + 94x^{4} - 2x^{3} - 13x^{2} - 77x + 47\end{aligned}$\end{center}\end{minipage} & 1 & $37^{11}$\\
   &            & \begin{minipage}{17.7cm}\begin{center}$\begin{aligned} x^{12} + x^{11} + x^{10} + 27x^{9} + 27x^{8} - 90x^{7} + 53x^{6} + 1353x^{5} + 768x^{4} - 3886x^{3} + 1600x^{2} + 5409x + 1847\end{aligned}$\end{center}\end{minipage} & 3 & $7^{8} 13^{11}$\\
   &            & \begin{minipage}{17.7cm}\begin{center}$\begin{aligned} x^{12} + x^{11} + 3x^{10} + 11x^{9} - 17x^{8} - 169x^{7} + 325x^{6} + 167x^{5} - 804x^{4} + 160x^{3} + 1102x^{2} - 780x + 1179\end{aligned}$\end{center}\end{minipage} & 1 & $61^{11}$\\\hline
12 & $\DD_6$ & \begin{minipage}{17.7cm}\begin{center}$\begin{aligned} x^{12} - 6x^{11} + 21x^{10} - 54x^{9} + 108x^{8} - 162x^{7} + 203x^{6} - 204x^{5} - 24x^{4} + 252x^{3} - 63x^{2} + 49\end{aligned}$\end{center}\end{minipage} & 1 & $2^{8} 3^{14} 7^{6}$\\
   &         & \begin{minipage}{17.7cm}\begin{center}$\begin{aligned} x^{12} + 13x^{10} + 58x^{8} + 109x^{6} + 86x^{4} + 25x^{2} + 1\end{aligned}$\end{center}\end{minipage} & 1 & $2^{16} 37^{6}$\\
\hline\end{tabular}
\newpage

\begin{tabular}{|>{\centering\arraybackslash}p{0.4cm}|>{\centering\arraybackslash}p{1.9cm}|>{\centering\arraybackslash}p{17.7cm}|c|>{\centering\arraybackslash}p{1.6cm}|}\hline
   $d$ & $G $ & $f$ & $h_L$ & $ \disc(L)$ \\\hline

12 & $\CC_2\times \CC_6$ & \begin{minipage}{17.7cm}\begin{center}$\begin{aligned} x^{12} - 10x^{10} + 76x^{8} - 224x^{6} + 496x^{4} - 192x^{2} + 64\end{aligned}$\end{center}\end{minipage} & 1 & $2^{18} 3^{6} 7^{8}$\\
   &                   & \begin{minipage}{17.7cm}\begin{center}$\begin{aligned} x^{12} - x^{11} - 3x^{10} + 7x^{9} + 5x^{8} - 33x^{7} + 13x^{6} - 132x^{5} + 80x^{4} + 448x^{3} - 768x^{2} - 1024x + 4096\end{aligned}$\end{center}\end{minipage} & 2 & $3^{6} 5^{6} 7^{10}$\\
   &                   & \begin{minipage}{17.7cm}\begin{center}$\begin{aligned} x^{12} + x^{11} - 7x^{10} - 8x^{9} + 34x^{8} + 42x^{7} - 76x^{6} - 147x^{5} + 76x^{4} + 20x^{3} - 154x^{2} + 638x + 841\end{aligned}$\end{center}\end{minipage} & 2 & $3^{6} 5^{6} 7^{10}$\\
   &                   & \begin{minipage}{17.7cm}\begin{center}$\begin{aligned} x^{12} + x^{11} - 16x^{10} - 27x^{9} + 92x^{8} + 133x^{7} - 42x^{6} - 224x^{5} + 241x^{4} - 895x^{3} + 900x^{2} + 88x + 181\end{aligned}$\end{center}\end{minipage} & 1 & $3^{6} 7^{8} 11^{6}$\\
   &                   & \begin{minipage}{17.7cm}\begin{center}$\begin{aligned} x^{12} - 15x^{10} + 92x^{8} - 302x^{6} + 596x^{4} - 792x^{2} + 841\end{aligned}$\end{center}\end{minipage} & 2 & $2^{12} 5^{6} 7^{10}$\\
   &                   & \begin{minipage}{17.7cm}\begin{center}$\begin{aligned} x^{12} - 6x^{11} + 39x^{10} - 140x^{9} + 447x^{8} - 1014x^{7} + 1867x^{6} - 2574x^{5} + 2703x^{4} - 2052x^{3} + 1077x^{2} - 348x + 53\end{aligned}$\end{center}\end{minipage} & 1 & $2^{12} 3^{16} 7^{6}$\\
   &                   & \begin{minipage}{17.7cm}\begin{center}$\begin{aligned} x^{12} - 12x^{10} + 54x^{8} - 130x^{6} + 213x^{4} - 198x^{2} + 361\end{aligned}$\end{center}\end{minipage} & 2 & $2^{12} 3^{18} 5^{6}$\\
\hline\end{tabular}
\newpage

\begin{tabular}{|>{\centering\arraybackslash}p{0.4cm}|>{\centering\arraybackslash}p{1.9cm}|>{\centering\arraybackslash}p{17.7cm}|c|>{\centering\arraybackslash}p{1.6cm}|}\hline
   $d$ & $G $ & $f$ & $h_L$ & $ \disc(L)$ \\\hline

 16 & $\CC_2^2\rtimes \CC_4$ & \begin{minipage}{17.7cm}\vspace{-1.00em}\begin{multline*}  x^{16} + 10x^{15} + 54x^{14} + 182x^{13} + 443x^{12} + 924x^{11} + 1972x^{10} + 3847x^{9} + 4642x^{8} + 738x^{7} - 4735x^{6} \\[-0.5em] - 784x^{5} + 14550x^{4} + 29227x^{3} + 29554x^{2} + 16791x + 5391\end{multline*}\end{minipage} & 1 & $13^{12} 17^{8}$\\\hline
16 & $\CC_2\times \CC_8$ & \begin{minipage}{17.7cm}\begin{center}$\begin{aligned} x^{16} + 1\end{aligned}$\end{center}\end{minipage} & 1 & $2^{64}$\\
   &                     & \begin{minipage}{17.7cm}\begin{center}$\begin{aligned} x^{16} - 8x^{14} + 44x^{12} - 128x^{10} + 270x^{8} - 288x^{6} + 216x^{4} - 32x^{2} + 4\end{aligned}$\end{center}\end{minipage} & 1 & $2^{62} 3^{8}$\\
   &                     & \begin{minipage}{17.7cm}\begin{center}$\begin{aligned} x^{16} - 16x^{14} + 104x^{12} - 352x^{10} + 660x^{8} - 672x^{6} + 336x^{4} - 64x^{2} + 2209\end{aligned}$\end{center}\end{minipage} & 2 & $2^{64} 5^{8}$\\\hline
16 & $\DD_8$ & \begin{minipage}{17.7cm}\vspace{-1.00em}\begin{multline*}  x^{16} - 4x^{15} + 27x^{14} - 84x^{13} + 237x^{12} - 336x^{11} + 308x^{10} - 134x^{9} \\[-0.5em] + 432x^{8} - 306x^{7} - 421x^{6} + 1317x^{5} - 36x^{4} - 805x^{3} + 295x^{2} + 475x + 125\end{multline*}\end{minipage} & 1 & $5^{8} 101^{8}$\\
   &         & \begin{minipage}{17.7cm}\begin{center}$\begin{aligned} x^{16} + 2x^{14} + 53x^{12} + 354x^{10} + 1160x^{8} + 1833x^{6} + 2729x^{4} + 2678x^{2} + 8281\end{aligned}$\end{center}\end{minipage} & 1 & $13^{8} 53^{8}$\\
\hline\end{tabular}
\newpage

\begin{tabular}{|>{\centering\arraybackslash}p{0.4cm}|>{\centering\arraybackslash}p{1.9cm}|>{\centering\arraybackslash}p{17.7cm}|c|>{\centering\arraybackslash}p{1.6cm}|}\hline
   $d$ & $G $ & $f$ & $h_L$ & $ \disc(L)$ \\\hline
20 & $\CC_{20}$ & \begin{minipage}{17.7cm}\begin{center}$\begin{aligned} x^{20} + x^{15} + x^{10} + x^{5} + 1\end{aligned}$\end{center}\end{minipage} & 1 & $5^{35}$\\\hline
20 & $\DD_{10}$ & \begin{minipage}{17.7cm}\vspace{-1em}\begin{multline*}x^{20} + 6x^{19} + 33x^{18} + 109x^{17} + 332x^{16} + 706x^{15} + 1299x^{14} + 1910x^{13} + 3303x^{12} + 7116x^{11} + 14445x^{10} \\[-0.5em] + 24009x^{9} + 30102x^{8} + 37094x^{7} + 54187x^{6} + 82991x^{5} + 119418x^{4} + 148247x^{3} + 185442x^{2} + 184250x + 112225\end{multline*}\end{minipage} & 1 & $19^{10} 43^{10}$\\
   &            & \begin{minipage}{17.7cm}\begin{center}$\begin{aligned} x^{20} - 55x^{16} + 1040x^{14} + 17860x^{12} + 74434x^{10} + 191605x^{8} + 137080x^{6} + 1950225x^{4} + 2274430x^{2} + 1493284\end{aligned}$\end{center}\end{minipage} & 2 & $2^{30} 5^{26} 7^{10}$\\\hline
20 & $\CC_2\times \CC_{10}$ & \begin{minipage}{17.7cm}\begin{center}$\begin{aligned} x^{20} + x^{19} - x^{17} - x^{16} + x^{14} + x^{13} - x^{11} - x^{10} - x^{9} + x^{7} + x^{6} - x^{4} - x^{3} + x + 1\end{aligned}$\end{center}\end{minipage} & 1 & $3^{10} 11^{18}$\\
   &                        & \begin{minipage}{17.7cm}\begin{center}$\begin{aligned} x^{20} - x^{18} + x^{16} - x^{14} + x^{12} - x^{10} + x^{8} - x^{6} + x^{4} - x^{2} + 1\end{aligned}$\end{center}\end{minipage} & 1 & $2^{20} 11^{18}$\\
   &                        & \begin{minipage}{17.7cm}\begin{center}$\begin{aligned} x^{20} - 9x^{18} + 53x^{16} - 182x^{14} + 454x^{12} - 711x^{10} + 796x^{8} - 469x^{6} + 190x^{4} - 15x^{2} + 1\end{aligned}$\end{center}\end{minipage} & 1 & $2^{20} 3^{10} 11^{16}$\\
\hline\end{tabular}
\newpage

\begin{tabular}{|>{\centering\arraybackslash}p{0.4cm}|>{\centering\arraybackslash}p{1.9cm}|>{\centering\arraybackslash}p{17.7cm}|c|>{\centering\arraybackslash}p{1.6cm}|}\hline
   $d$ & $G $ & $f$ & $h_L$ & $ \disc(L)$ \\\hline

32 & $\CC_2^2\rtimes \CC_8$ & \begin{minipage}{17.7cm}\vspace{-1em}\begin{multline*} x^{32} - 4x^{31} + 10x^{30} + 150x^{28} - 484x^{27} + 1290x^{26} + 332x^{25} + 6797x^{24} - 10156x^{23} + 32928x^{22} + 45672x^{21} \\[-0.5em] + 150325x^{20} - 54628x^{19} + 268716x^{18} + 645284x^{17} + 1404151x^{16} - 448940x^{15} - 2976140x^{14} \\[-0.5em] - 4612624x^{13} - 599552x^{12} + 4399140x^{11} + 6786614x^{10} + 3071008x^{9} - 884241x^{8} - 2543532x^{7} \\[-0.5em] - 1018052x^{6} + 30728x^{5} + 369232x^{4} + 151840x^{3} + 201784x^{2} + 149760x + 67472\end{multline*}\end{minipage} & 1 & $2^{48} 17^{28}$\\\hline
32 & $\CC_2\times \CC_4\rtimes \CC_4$ & \begin{minipage}{17.7cm}\vspace{-1em}\begin{multline*} x^{32} - 8x^{30} + 20x^{28} - 96x^{26} + 248x^{24} + 1256x^{22} - 284x^{20} + 4320x^{18} - 31058x^{16} - 4344x^{14} \\[-0.5em] + 164812x^{12} - 147392x^{10} - 54376x^{8} - 138408x^{6} + 333116x^{4} - 49408x^{2} + 5329\end{multline*}\end{minipage} & 1 & $2^{108} 3^{24}$\\\hline
32 & $\CC_2\times (\CC_4 . \CC_4)$ & \begin{minipage}{17.7cm}\vspace{-1em}\begin{multline*} x^{32} - 8x^{31} + 40x^{30} - 180x^{29} + 694x^{28} - 2312x^{27} + 6908x^{26} - 18252x^{25} + 42936x^{24} - 89856x^{23} + 166568x^{22} \\[-0.5em] - 271092x^{21} + 378194x^{20} - 434016x^{19} + 364428x^{18} - 112172x^{17} - 301690x^{16} + 734040x^{15} - 939336x^{14} \\[-0.5em] + 732052x^{13} - 175390x^{12} - 395304x^{11} + 619428x^{10} - 445780x^{9} + 123392x^{8} \\[-0.5em] + 68432x^{7} - 80520x^{6} + 29684x^{5} + 6838x^{4} - 7696x^{3} + 3604x^{2} - 20x + 241\end{multline*}\end{minipage} & 1 & $2^{80} 5^{28}$\\
\hline\end{tabular}
\newpage

\begin{tabular}{|>{\centering\arraybackslash}p{0.4cm}|>{\centering\arraybackslash}p{1.9cm}|>{\centering\arraybackslash}p{17.7cm}|c|>{\centering\arraybackslash}p{1.6cm}|}\hline
   $d$ & $G $ & $f$ & $h_L$ & $ \disc(L)$ \\\hline
   48 & $\DD_4\rtimes_* \mathfrak S_3$ & \begin{minipage}{17.7cm}\vspace{-1em}\begin{multline*} x^{48} - 16x^{47} + 154x^{46} - 1062x^{45} + 5861x^{44} - 27032x^{43} + 108118x^{42} - 382247x^{41} + 1216320x^{40} - 3519040x^{39} \\[-0.5em]+ 9354487x^{38} - 22981410x^{37} + 52537749x^{36} - 112139241x^{35} + 224599406x^{34} - 422866886x^{33} + 751689957x^{32} \\[-0.5em]- 1262741576x^{31} + 2014037588x^{30} - 3052638118x^{29} + 4423121908x^{28} - 6134222492x^{27} + 8194638050x^{26} \\[-0.5em]- 10554045180x^{25} + 13166769236x^{24} - 15909389418x^{23} + 18637913584x^{22} - 21183975650x^{21} \\[-0.5em]+ 23316032723x^{20} - 24826123460x^{19} + 25455052543x^{18} - 25225507043x^{17} + 24074543095x^{16} \\[-0.5em]- 21881216430x^{15} + 19020997095x^{14} - 15911316338x^{13} + 12715024536x^{12} - 9328992351x^{11} \\[-0.5em]+ 6682355760x^{10} - 4672449816x^{9} + 2802278711x^{8} - 1633037848x^{7} + 958754217x^{6} - 529620206x^{5} \\[-0.5em]+ 181598936x^{4} - 53495044x^{3} + 99083410x^{2} + 28884359x + 7724261\end{multline*}\end{minipage} & 1 & $2^{32} 5^{24} 101^{24}$\\

\hline\end{tabular}
\newpage
\end{center}\end{landscape}
\clearpage{}%

\bibliographystyle{amsplain}
\bibliography{cmfields}
\end{document}